\documentclass[preprint,nopreprintline,12pt]{elsarticle}

\usepackage{amssymb}
\usepackage{amsmath}
\allowdisplaybreaks
\usepackage{amsthm}

\usepackage{ifthen}

\usepackage{appendix}

\usepackage{nicematrix}
\usepackage{array}
\usepackage{multirow}

\usepackage{titlesec}

\titleformat{\section}
  {\normalfont\bfseries}
  {\thesection}{1em}{}        

\titleformat{\subsection}
  {\normalfont}
  {\thesubsection}{1em}{}

\titleformat{\subsubsection}
  {\normalfont}
  {\thesubsubsection}{1em}{}

\numberwithin{equation}{section}

\newtheorem{Definition}{Definition}[section]

\newtheorem{Theorem}[Definition]{Theorem}
\newtheorem{Proposition}[Definition]{Proposition}

\newtheorem{Corollary}[Definition]{Corollary}

\newtheorem{Fact}[Definition]{Fact}

\newcounter{arxiv}
\ifthenelse{\value{arxiv} = 0}{
\input 00-01_specific_journal_20260130a.tex
} {
\usepackage{hyperref}
}

\begin{document}

\begin{frontmatter}



\title{Unique direct-sum decompositions of
infinite-dimensional unitary representations of the infinite dihedral group via real CMV matrices}


\author[melco,hu]{Yuki Fujii\corref{cor1}} %% Author name
\ead{Fujii.Yuki@df.MitsubishiElectric.co.jp}
\cortext[cor1]{Corresponding author.}

\affiliation[melco]{organization={Mitsubishi Electric Corporation, Information Technology R\&D Center},%Department and Organization
            addressline={5-1-1 Ofuna}, 
            city={Kamakura-shi},
            postcode={247-8501}, 
            state={Kanagawa},
            country={Japan}}
\affiliation[hu]{organization={Graduate School of Advanced Science and Engineering, Hiroshima University},%Department and Organization
            addressline={1-3-1 Kagamiyama}, 
            city={Higashi-Hiroshima City},
            postcode={739-8526}, 
            state={Hiroshima},
            country={Japan}}

%% Abstract
\begin{abstract}
Unitary representations of the infinite dihedral group arise naturally and
play an important role in a wide range of areas including mathematics, physics,
and engineering.
CMV matrices, introduced by Cantero, Moral, and Vel\'azquez, are five-diagonal
unitary matrices that are standard tools in the theory of orthogonal
polynomials on the unit circle.

It is known that every unitary representation of the infinite dihedral group on
a separable Hilbert space admits a discrete direct-sum decomposition into
representations arising from the natural factorizations of real CMV matrices
and copies of two exceptional one-dimensional irreducible representations. 

We prove that such a decomposition is unique up to unitary equivalence when
the summands arising from real CMV matrices are grouped according to their
multiplicities and the spectral measures of the corresponding matrices
are mutually singular across distinct multiplicities.
Furthermore, we show that the representation associated with each such
real CMV matrix $C$ admits a direct integral decomposition into irreducible
unitary representations with respect to the push-forward of the restriction
of the spectral measure of $C$ to the closed upper half circle.

These results make it possible to study the structure of unitary
representations of the infinite dihedral group using CMV matrices and the
theory of orthogonal polynomials on the unit circle.
\end{abstract}


\ifthenelse{\value{arxiv} = 0}{
\input 00-02_specific_journal_20261026a.tex
} {}

%% Keywords

\end{frontmatter}

\ifthenelse{\value{arxiv} = 1}{
\tableofcontents
} {}


\section{Introduction and main results}
\label{sec:intro}

The theory of pairs of orthogonal projections on a Hilbert space, often
called two projections theory, has been extensively studied since the
pioneering work of C.~Jordan in 1875~\cite{CJordan1875}.
For an overview of the theory, we refer the reader to
the survey~\cite{ABottcher2010} and the references therein.

Pairs of orthogonal projections arise naturally in integral equations and
numerical analysis~\cite{SRoch1988, CDavis1970, FDeutsch1984},
the geometry of Grassmann manifolds~\cite{SBerceanu1995},
probability, statistics, and machine learning~\cite{AGoldman2022, QQin2024, KDCock2002, YAizenbud2025},
control theory and signal processing~\cite{PVOverschee1993, JMSchumacher1992, BAdcock2015, DAttali2026},
and quantum information processing~\cite{Primaatmaja2023, MMcKague2012}.
They therefore play an important role in a wide range of areas including
mathematics, physics, and engineering.

Specifying a pair of orthogonal projections $P$ and $Q$ on a Hilbert space $V$
is equivalent to specifying a unitary representation $\pi$ of the infinite
dihedral group $D_\infty$, generated by $g_1$ and $g_2$, via
$\pi(g_1)=2P-\operatorname{id}_V$ and
$\pi(g_2)=2Q-\operatorname{id}_V$.
Thus, questions about pairs of orthogonal projections, such as determining the spectra
of operators generated by $P$ and $Q$, can be reformulated as questions
about unitary representations of $D_\infty$.
In this paper, we study two projections theory from the viewpoint of
the representation theory of $D_\infty$.

A starting point for our main results is the following unique discrete
direct-sum decomposition theorem for unitary representations of
$\mathbb{Z}$.
Its summands can be realized by CMV matrices~\cite{MJCantero2003}, which are five-diagonal
unitary matrices with a cyclic vector.
From the viewpoint of group representation theory,
a unitary representation of $\mathbb{Z}$
whose generator is represented by a CMV matrix
is a spherical representation of the Gelfand pair $(\mathbb{Z},\{e\})$.
For a Gelfand pair $(G,K)$, we follow \cite{ABorodin2017}
in defining a spherical representation as a unitary
representation admitting a cyclic $K$-fixed vector,
without assuming irreducibility.
CMV matrices play a fundamental role in the theory of orthogonal
polynomials on the unit circle; see, e.g., \cite{BSimon2005p1,BSimon2005p2}.
For notational convenience, we include the unique $0 \times 0$ matrix,
acting on the zero-dimensional Hilbert space $\{0\}$, among CMV matrices.
We also regard the zero representation as cyclic, with cyclic vector $0$;
the corresponding spectral measure is the zero measure.
Throughout this paper, all unitary representations are assumed to act on
separable Hilbert spaces.

\begin{Fact}[Unique CMV decomposition for $\mathbb{Z}$]
\label{fact:zunitaryrepresentationdecomposition}
Let $\pi$ be a unitary representation of $\mathbb{Z}$.
Then the following statements hold.
\begin{enumerate}
  \item[(i)] There exists a sequence
  $\rho_1,\rho_2,\dots,\rho_\infty$ of cyclic unitary representations
  of $\mathbb{Z}$ such that the following unitary equivalence holds:
  \begin{equation*}
  \pi \simeq
  (\rho_1 \oplus 2\rho_2 \oplus \cdots)
  \oplus \infty\rho_\infty.
  \end{equation*}
    Here, $m\rho_m$ denotes the discrete direct sum of $m$ copies of
  $\rho_m$, with $m=\infty$ denoting countably infinitely many copies.

Each $\rho_i(1)$ is represented by a CMV matrix with respect to a complete
orthonormal system.
Let $\mu_i$ denote its spectral measure with respect to a cyclic vector.
Then $\mu_k \perp \mu_l$ for any distinct $k, l \in \mathbb{N} \cup \{\infty\}$.
  \item[(ii)] The sequence $\rho_1,\rho_2,\dots,\rho_\infty$
  satisfying~(i) is uniquely determined up to unitary equivalence.
\end{enumerate}
\end{Fact}

Furthermore, the spectral resolution theorem for unitary operators yields a direct
integral decomposition of each summand in
Fact~\ref{fact:zunitaryrepresentationdecomposition} into irreducible
unitary representations, as follows.

\begin{Fact}[Direct integral decomposition of CMV matrices]
\label{fact:zunitaryrepresentationdirectintegraldecomposition}
Let $C$ be a CMV matrix, let $\rho$ be the unitary representation
of $\mathbb{Z}$ defined by $\rho(1)=C$, and let
$\mu$ be the spectral measure of $\rho(1)$ with respect to a cyclic vector.
Then
\begin{equation*}
\rho \simeq \int_{S^1}^{\oplus} \xi(z)\,d\mu(z).
\end{equation*}
Here, for $z\in S^1$, $\xi(z)$ is the one-dimensional irreducible unitary
representation of $\mathbb{Z}$ in which the generator $1$ acts on
$\mathbb{C}$ by multiplication by $z$.
\end{Fact}

We identify the subgroup $\langle g_2g_1 \rangle$ of $D_\infty$ with
$\mathbb{Z}$.
Under this identification, $D_\infty$ is an index-two extension of
$\mathbb{Z}$.
It is therefore natural to expect a close connection between the
unitary representation theories of the two groups.

In fact, every spherical unitary representation of the Gelfand pair
$(D_\infty,\langle g_1 \rangle)$ 
can be realized, up to unitary equivalence, through the natural factorization of a
real CMV matrix. This characterization was proved for infinite-dimensional
representations in~\cite{YFujii2026}, and the finite-dimensional case can be established
by the same argument.
To describe this realization, let $C = LM$ be the natural factorization
of a real CMV matrix $C$. Since both $L$ and $M$ are unitary involutions, the assignments
\begin{equation*}
\pi(g_1)=M,\qquad \pi(g_2)=L
\end{equation*}
define a unitary representation of $D_\infty$ with $\pi(g_2g_1)=C$.
Following our previous work~\cite{YFujii2026},
we call any finite- or infinite-dimensional representation obtained in this way
an OSF (one-shifted form) representation.
We also regard the zero representation as an OSF representation.

The following existence theorem was established
in our previous work~\cite{YFujii2026}.

\begin{Fact}[Existence of OSF decomposition for $D_\infty$ \cite{YFujii2026}]
\label{fact:introgeneralizedjordanlemma}
Every unitary representation of $D_\infty$ is unitarily equivalent to an
at most countable discrete direct sum whose summands are either copies of
the exceptional one-dimensional representations
$\iota_{-,+},\iota_{-,-}$ or OSF representations of dimension at most two
or of infinite dimension.
Here, $\iota_{-,\pm}$ are the one-dimensional representations defined by
$\iota_{-,\pm}(g_1)=-1$ and $\iota_{-,\pm}(g_2)=\pm1$.
\end{Fact}

The main aim of this paper is to establish decomposition results
for unitary representations of $D_\infty$ in analogy with
the CMV decomposition results for $\mathbb{Z}$ recalled in
Fact~\ref{fact:zunitaryrepresentationdecomposition} and
Fact~\ref{fact:zunitaryrepresentationdirectintegraldecomposition}.
More precisely, we seek to establish a uniqueness criterion for OSF
decompositions in terms of the mutual singularity of the spectral
measures of the corresponding real CMV matrices.
We also aim to describe the direct integral decomposition
of each OSF representation
into irreducible unitary representations
in terms of the spectral measure
of its associated CMV matrix.
Our approach combines the characterization of OSF representations
as spherical representations, in both finite and infinite dimensions,
with standard results on direct-sum and direct integral
decompositions of unitary representations of type I groups
 and the general theory of spherical representations.

The first two of our three main contributions are stated below.

\begin{Theorem}[Unique OSF decomposition for $D_\infty$]
\label{theorem:introdinftyunitaryrepresentationdecomposition}
Let $\pi$ be a unitary representation of $D_\infty$.
Then the following statements hold.
\begin{enumerate}
  \item[(i)] There exist a sequence
  $\rho_1,\rho_2,\dots,\rho_\infty$ of OSF representations,
  and unitary representations $\iota_+$ and $\iota_-$
  such that
  \begin{equation*}
  \pi \simeq
  (\rho_1 \oplus 2\rho_2 \oplus \cdots) \oplus \infty\rho_\infty
  \oplus \iota_{+} \oplus \iota_{-},
  \end{equation*}
  where $\iota_{+}$ and $\iota_{-}$ are at most countable discrete direct sums
  of copies of $\iota_{-,+}$ and $\iota_{-,-}$, respectively.
  For each $\rho_i$,
  let $\mu_i$ denote a spectral measure of the real CMV matrix
  $\rho_i(g_2g_1)$ with respect to a cyclic vector.
  Then $\mu_k \perp \mu_l$ for any distinct $k, l \in \mathbb{N} \cup \{\infty\}$.

  \item[(ii)] The sequence $\rho_1,\rho_2,\dots,\rho_\infty, \iota_{+}, \iota_{-}$
  satisfying~(i) is uniquely determined up to unitary equivalence.
\end{enumerate}
\end{Theorem}

\begin{Theorem}[Direct integral decomposition of OSF representations]
\label{theorem:introductionexplicitdirectintegralosf}
Let $\pi$ be an OSF representation,
and let $\mu$ denote the spectral measure of $\pi(g_2g_1)$
with respect to a cyclic vector.
Then $\pi$
admits the following direct integral decomposition into irreducible
unitary representations:
\begin{equation*}
\pi \simeq \int_{\widehat{D_\infty}}^{\oplus} \xi\,dm(\xi),
\end{equation*}
where $m$ is the push-forward of the restriction of $\mu$ to the closed upper
half circle $S^1_+$ by the natural map from $S^1_+$ to the unitary
dual $\widehat{D_\infty}$.
\end{Theorem}

Theorems~\ref{theorem:introdinftyunitaryrepresentationdecomposition}
and~\ref{theorem:introductionexplicitdirectintegralosf} are discussed in
detail in
Section~\ref{sec:uniqueosfdecomposition}
and Section~\ref{sec:directintegraldecomposition}, respectively.
These results make it possible to study the structure of unitary
representations of $D_\infty$ using CMV matrices and the theory of
orthogonal polynomials on the unit circle.

Furthermore,
Theorem~\ref{theorem:introductionexplicitdirectintegralosf} shows that,
within the class of OSF representations, the unitary equivalence class of
the restriction $\pi|{\mathbb{Z}}$ uniquely determines the unitary
equivalence class of the original representation $\pi$ of $D_\infty$.
In other words, we obtain the following corollary.

\begin{Corollary}
Let $\pi_1$ and $\pi_2$ be OSF representations.
If $\pi_1|{\mathbb{Z}} \simeq \pi_2|{\mathbb{Z}}$, then
$\pi_1 \simeq \pi_2$.
\end{Corollary}

Using the functional realization of OSF representations on
$L^2(S^1,\mu)$, where $\mu$ is the spectral measure of $\pi(g_2g_1)$
with respect to a suitable cyclic vector,
we explicitly construct a unitary intertwining operator between
two OSF representations of $D_\infty$ from a unitary
intertwining operator between their restrictions to $\mathbb{Z}$.
This functional realization is a reformulation of
a well-known result in the theory of
orthogonal polynomials on the unit circle~\cite{BSimon2005p1}.
Furthermore, for OSF representations, we use the same realization
to explicitly obtain the standard $2 \times 2$ block-operator
representation of a pair of orthogonal projections appearing
in Halmos' two projections theorem \cite{PRHalmos1969}.
These results are established in Section~\ref{sec:realizationfunctionspace}.

The third contribution shows that the commutant of an OSF representation $\pi$
can also be described explicitly in terms of its restriction to
$\mathbb{Z}$, as a von Neumann algebra generated by $\pi(g_2g_1)+\pi(g_2g_1)^*$.
For details, see Section~\ref{sec:directintegraldecomposition}.

\begin{Theorem}[Commutant structure of OSF representations]
\label{theorem:intronontrivialintertwiningoperator}
Let $(\pi,V)$ be an OSF representation. Then
\begin{equation*}
\pi(D_\infty)'=W^*(\pi(g_2g_1)+\pi(g_1g_2)),
\end{equation*}
where $W^*(\pi(g_2g_1)+\pi(g_1g_2))$ denotes the von Neumann algebra
generated by $\pi(g_2g_1)+\pi(g_1g_2)$ and the identity operator.
\end{Theorem}

In Section~\ref{sec:multiplicityfreerepresentationcharacterization},
we further investigate the structure of multiplicity-free unitary
representations of $D_\infty$.
In particular, we show that the equality in
Theorem~\ref{theorem:intronontrivialintertwiningoperator}
characterizes unitary representations of $D_\infty$
whose restrictions to $\mathbb{Z}$ are multiplicity-free.

\section{Preliminaries}
\label{sec:preliminary}

\subsection{Spherical representations}

The following direct integral decomposition formula for spherical representations
is well known; see, e.g., \cite{JFaraut2008}.

\begin{Fact}
\label{fact:spectralmeasurespherical}
Let $(G,K)$ be a Gelfand pair with $G$ a
second countable locally compact group of type I,
let $\Omega$ denote the spherical dual of $(G,K)$,
that is, the set of unitary equivalence classes of
irreducible spherical representations of $(G,K)$,
and let $\pi$ be a spherical representation of $(G,K)$.
Then there exist a finite positive regular Borel measure $m$ on
$\Omega$ and a unitary intertwining operator
$\Phi:\pi \to \int_{\Omega}^\oplus \xi dm(\xi)$
such that $\Phi \circ \pi(G)' \circ \Phi^{-1} = \{ M_f \mid f \in L^\infty(\Omega, m)\}$.
Here $M_f$ denotes multiplication by $f$.
\end{Fact}

The following definition is based on \cite{JFaraut2008}.

\begin{Definition}[Spectral measure of spherical representations]
Under the assumptions and notation of Fact~\ref{fact:spectralmeasurespherical},
we call $m$ a spectral measure of $\pi$.
\end{Definition}

The following fact is a consequence of standard results on
direct integrals and multiplicity-free unitary representations
of locally compact groups; see, e.g., \cite{JDixmier1981}.
Since we could not find a reference stating it in the precise
form needed here, we provide a proof in
\ref{sec:proofofsphericalcriteria} for completeness.

\begin{Fact}
\label{fact:sphericalcriteria}
Let $G$ be a second countable locally compact group of type I,
let $K$ be a compact subgroup of $G$,
let $\widehat{G}^{K\text{-}\mathrm{sph}}$ denote $\{ \xi \in \widehat{G} \mid \text{$\xi$
has a nonzero $K$-fixed vector}\}$.
Assume that $\widehat{G}^{K\text{-}\mathrm{sph}}$ is a Borel subset of $\widehat{G}$.
Let $\pi$ be a multiplicity-free unitary representation of $G$ such that
$\pi$ admits a direct integral decomposition
$\pi \simeq \int_{\widehat{G}^{K\text{-}\mathrm{sph}}}^\oplus \pi_\omega dm(\omega)$
for some positive finite regular Borel measure $m$ on $\widehat{G}^{K\text{-}\mathrm{sph}}$.
Then $\pi$ has a $K$-fixed cyclic vector.
\end{Fact}

\subsection{Unitary representations of $D_\infty$}
\label{subsec:preliminaryinfinitedihedralgroup}

We first recall the irreducible unitary representations of $D_\infty$.

\begin{Definition}[Irreducible representations of $D_\infty$]
Define unitary representations $\iota_{+,+}$, $\iota_{+,-}$, $\iota_{-,+}$, $\iota_{-,-}$ of $D_\infty$
on $\mathbb{C}$
by
$\iota_{+,\pm}(g_1)=1$, $\iota_{-,\pm}(g_1)=-1$,
$\iota_{\pm,+}(g_2)=1$, $\iota_{\pm, -}(g_2)=-1$.

Define a unitary representation $\eta_\theta$ $(\theta \in (0,\pi))$
 of $D_\infty$
on $\mathbb{C}^2$
by
\begin{equation*}
\eta_\theta(g_1)=
\begin{pmatrix}
1 & 0\\
0 & -1
\end{pmatrix},\quad
\eta_\theta(g_2)=
\begin{pmatrix}
\cos\theta & \sin\theta\\
\sin\theta & -\cos\theta
\end{pmatrix}.
\end{equation*}
\end{Definition}

The following fact is a standard consequence of the general
theory of unitary representations of $D_\infty$;
see, e.g., \cite{EThoma1964, ABottcher2010, HBehncke1970, HBehncke1971}.

\begin{Fact}[Classification of irreducible unitary representations of $D_\infty$]
\label{fact:behncke}
The group $D_\infty$ is of type I, and
\begin{equation*}
\widehat{D_\infty}=\{\iota_{+,+},\iota_{+,-},\iota_{-,+},\iota_{-,-}\}
\cup \{\eta_\theta \mid \theta \in (0,\pi)\}.
\end{equation*}
\end{Fact}

The following fact follows from elementary calculations.

\begin{Fact}
\label{fact:sphericaldual}
$(D_\infty,\langle g_1 \rangle)$ is a Gelfand pair, and
the spherical dual of $(D_\infty,\langle g_1 \rangle)$ is
$\widehat{D_\infty}\setminus \{\iota_{-,+},\iota_{-,-}\}$.
\end{Fact}

The following characterization was established for
infinite-dimensional representations in \cite{YFujii2026},
and the same argument applies to finite-dimensional
representations.

\begin{Theorem}
\label{theorem:sphericalrepresentation}
Let $\pi$ be a unitary representation of $D_\infty$.
Then $\pi$ is a spherical representation of $(D_\infty,\langle g_1 \rangle)$
if and only if $\pi$ is unitarily equivalent to an OSF representation.
\end{Theorem}

The following fact is a reformulation of a well-known result in the theory of
orthogonal polynomials on the unit circle \cite{BSimon2005p1}.
Note that every nonzero CMV matrix has $e_1$ as a cyclic vector,
where $e_1$ is the vector in the corresponding Hilbert space
whose first component is $1$ and whose other components are $0$.

\begin{Fact}[Functional realization of OSF representations]
\label{fact:osffunctionalrealization}
Let $(\pi,V)$ be a nonzero OSF representation
and let $\mu$ be a spectral measure of $\pi(g_2g_1)$ with respect to $e_1$.
Then there exists a unitary intertwining operator $\Phi:V \to L^2(S^1,\mu)$ such that
\begin{equation*}
\Phi \circ \pi(g_2g_1) \circ \Phi^{-1}=M_{\boldsymbol{z}},
\quad
\Phi \circ \pi(g_1) \circ \Phi^{-1} = c,
\end{equation*}
where $\boldsymbol{z}$ is defined by $\boldsymbol{z}(z)=z$ for $z\in S^1$,
and $c$ is defined by $[c(f)](z)=f(\bar{z})$ $(f \in L^2(S^1,\mu),\ z \in S^1)$.
\end{Fact}

\section{Unique OSF decomposition for $D_\infty$}
\label{sec:uniqueosfdecomposition}

We establish the following unique discrete direct-sum decomposition
of unitary representations of $D_\infty$ into OSF representations
and copies of the two exceptional one-dimensional representations
$\iota_{-,+}$ and $\iota_{-,-}$,
in terms of the mutual singularity of the spectral measures of the corresponding
real CMV matrices. 

\begin{Theorem}[Unique OSF decomposition for $D_\infty$]
\label{theorem:dinftyunitaryrepresentationdecomposition}
Let $(\pi,V)$ be a unitary representation of $D_\infty$.
Then there exist a sequence
$\rho_1,\rho_2,\dots,\rho_\infty$ of OSF representations
and a pair $(\iota_+,\iota_-)$ of unitary representations,
uniquely determined up to unitary equivalence,
satisfying the following conditions:
\begin{itemize}
\item The following decomposition holds:
\begin{equation*}
\pi \simeq \bigoplus_{i=1}^\infty i\rho_i \oplus \infty \rho_\infty \oplus \iota_+ \oplus \iota_-,
\end{equation*}
where $i \rho$ $(i \in \mathbb{N} \cup \{\infty\})$
is the discrete direct sum of $i$ copies of $\rho$.

\item For each nonzero $\rho_i$, let $\mu_i$ denote
the spectral measure of the real CMV matrix $\rho_i(g_2g_1)$ with respect to $e_1$.
Then these measures are mutually singular.

\item $\iota_+$ and $\iota_-$ are at most countable discrete direct sums
of copies of $\iota_{-,+},\iota_{-,-}$, respectively.
\end{itemize}
\end{Theorem}
\begin{proof}
Since $D_\infty$ is a type I group (Fact~\ref{fact:behncke}),
$\pi$ admits a unique decomposition into
mutually disjoint multiplicity-free representations
grouped according to their multiplicities
(see, e.g., \cite[Theorem~7.32]{GBFolland2016}).

From Fact~\ref{fact:sphericaldual}, all irreducible
unitary representations of $D_\infty$, except
$\iota_{-,+}$ and $\iota_{-,-}$, are spherical.
Separating these two exceptional representations,
Fact~\ref{fact:sphericalcriteria} and
Theorem~\ref{theorem:sphericalrepresentation}
show that each remaining multiplicity-free
component is unitarily equivalent to an OSF representation.

From Theorem~\ref{theorem:explicitdirectintegralosf},
a spectral measure of each nonzero OSF component
on the unitary dual is equivalent to the
push-forward of $\mu|S^1_+$ by $\gamma$, 
where $\mu$ is the spectral measure of
its corresponding CMV matrix with respect to $e_1$.

The spectral measures of nonzero real CMV matrices with respect to $e_1$
are conjugate-invariant (see, e.g., \cite{BSimon2005p1});
that is, $\mu(E)=\mu(\bar{E})$ for every Borel subset
$E$ of $S^1$, where $\bar{E}=\{\bar{z} \mid z\in E\}$.
Since $\gamma$ is a homeomorphism onto the
spherical dual (Theorem~\ref{theorem:explicitdirectintegralosf}),
mutual singularity of these measures is preserved in both directions.

Therefore, the existence and uniqueness of the asserted
decomposition follow from the uniqueness of
the decomposition described at the beginning of the proof.
\end{proof}

\section{Direct integral decomposition of OSF representations}
\label{sec:directintegraldecomposition}

The following theorem shows that a spectral measure of an OSF representation $\pi$
can be chosen as the push-forward of the spectral measure
of the real CMV matrix $\pi(g_2g_1)$ restricted to $S^1_+$
under a homeomorphism from $S^1_+$ onto the spherical dual
of $(D_\infty,\langle g_1\rangle)$.

The theory of orthogonal polynomials on the unit circle provides
a wealth of results concerning the spectral measures of CMV matrices.
For instance, when the Verblunsky coefficients form a constant sequence,
the associated spectral measure can be computed explicitly.
For details, see, e.g., \cite{BSimon2005p1,BSimon2005p2}.
Consequently, Theorem~\ref{theorem:explicitdirectintegralosf}
indicates that these tools from the theory of orthogonal polynomials on
the unit circle can be used to study the structure of OSF
representations.

\begin{Theorem}[Direct integral decomposition of OSF representations]
\label{theorem:explicitdirectintegralosf}
Let $(\pi,V)$ be a nonzero OSF representation,
let $\mu$ be the spectral measure of the real CMV matrix $\pi(g_2g_1)$ with respect to $e_1$,
and let $m$ be the push-forward measure of $\mu|S^1_+$
on $\widehat{D_\infty}$ by a homeomorphism $\gamma$
from $S^1_+$ to the spherical dual $\widehat{D_\infty}^{\mathrm{sph}}$ of
$(D_\infty,\langle g_1 \rangle)$,
where $\gamma(\pm 1)=\iota_{+,\pm}$ and
$\gamma(e^{i\theta})=\eta_\theta$ for $\theta \in (0,\pi)$.

Then the following unitary equivalence holds:
\begin{equation*}
\pi \simeq \int_{\widehat{D_\infty}}^\oplus \xi\ dm(\xi).
\end{equation*}
Furthermore, $m(\{\iota_{-,+},\iota_{-,-}\})=0$
and
\begin{equation*}
\pi(D_\infty)'=W^*(\pi(g_2g_1)+\pi(g_1g_2)).
\end{equation*}
\end{Theorem}
\begin{proof}
From elementary calculations, $\gamma$ is a homeomorphism.
Let $C$, $T$, and $A$ be defined by $C = \pi(g_2g_1)$,
$T = C + C^*$, and $A = T/2$,
and define
$r: \widehat{D_\infty}^{\mathrm{sph}} \to [-1,1]$ by
$r(\xi) = \operatorname{Re}(\gamma^{-1}(\xi))$.
From Theorem~\ref{theorem:sphericalrepresentation}
and Fact~\ref{fact:spectralmeasurespherical}, there exist a finite
positive regular Borel measure $m'$ on $\widehat{D_\infty}^{\mathrm{sph}}$
and a unitary intertwining operator
$\Phi:\pi \to \int_{\widehat{D_\infty}^{\mathrm{sph}}}^{\oplus} \xi\ dm'(\xi)$
such that
$\Phi \pi(D_\infty)' \Phi^{-1} =\{M_f \mid f \in L^\infty(\widehat{D_\infty}^{\mathrm{sph}},m')\}$.

Since
$\Phi (T/2) \Phi^{-1} = M_r$
and $r$ is a homeomorphism,
the Borel functional calculus gives
$\pi(D_\infty)' = W^*(T)$.

Let $\eta: S^1 \to [-1,1]$ be defined by
$\eta(z) = \operatorname{Re}(z)$.
We prove that $r_*m' \sim \eta_*\mu$.
Let $E_A$ denote the projection-valued spectral measure of $A$,
and let $B$ be a Borel subset of $[-1,1]$.
From the spectral resolution theorem for bounded self-adjoint operators,
\begin{equation*}
\Phi E_A(B) \Phi^{-1}=\Phi 1_B(A) \Phi^{-1}
=1_B(M_{r})=M_{1_B \circ r}=M_{1_{r^{-1}(B)}}.
\end{equation*}
Hence $(r_*m')(B)=0 \iff E_A(B)=0$.
Since $e_1$ is cyclic for $C$ and
$E_A(B)$ commutes with $C$, we obtain
$E_A(B)=0 \iff (E_A(B)e_1,e_1) = 0$.
Indeed,
since $(E_A(B)e_1,e_1) = \|E_A(B)e_1\|^2$,
if
$(E_A(B)e_1,e_1) = 0$
then $E_A(B)e_1 = 0$. Hence
$E_A(B)C^ne_1 = 0$ for every $n\in\mathbb{Z}$.
The cyclicity of $e_1$ for $C$ then gives
$E_A(B)=0$.
Since $A=\eta(C)$, we have
$(E_A(B)e_1,e_1)=(\eta_*\mu)(B)$.
It follows that
$r_*m' \sim \eta_*\mu$.

Since $\mu$ is conjugate-invariant,
$\eta_*\mu\sim\eta_*(\mu|{S^1_+})$.
Using $r\circ\gamma=\eta|{S^1_+}$
and the fact that $r$ is a homeomorphism, we obtain
\begin{equation*}
m'
\sim(r^{-1})_*(\eta_*\mu)
\sim\gamma_*(\mu|{S^1_+})
=m.
\end{equation*}
Hence
\begin{equation*}
\pi \simeq
\int_{\widehat{D_\infty}}^\oplus
\xi\ dm(\xi).
\end{equation*}
Finally,
$m(\{\iota_{-,+},\iota_{-,-}\})=0$
follows from the definition of $\gamma$.
\end{proof}

\section{Functional realization of OSF representations}
\label{sec:realizationfunctionspace}

Using the functional realization in Fact~\ref{fact:osffunctionalrealization},
we explicitly construct a unitary intertwining operator between two OSF representations of
$D_\infty$ from a unitary intertwining operator between their restrictions to $\mathbb{Z}$.

\begin{Proposition}
\label{proposition:osfintertwinerlifting}
Let $\pi_i$ $(i=1,2)$ be the functional realizations
of two nonzero OSF representations on $L^2(S^1,\mu_i)$
given by Fact~\ref{fact:osffunctionalrealization}.
Let
\begin{equation*}
U:L^2(S^1,\mu_1)\to L^2(S^1,\mu_2)
\end{equation*}
be a unitary intertwining operator from
$\pi_1|{\mathbb{Z}}$ to $\pi_2|{\mathbb{Z}}$.
Define $h=U1$.
Then the operator
\begin{equation*}
W(f)=|h|f\quad (f \in L^2(S^1,\mu_1))
\end{equation*}
is a unitary intertwining operator from
$\pi_1$ to $\pi_2$.
\end{Proposition}
\begin{proof}
Since $U$ intertwines the multiplication operators
$M_{\boldsymbol{z}}$, the spectral resolution theorem for unitary operators
gives
$Uf = hf$, where $h = U1$.
The unitarity of $U$ implies
$d\mu_1 = |h|^2\ d\mu_2$ and $\mu_1 \sim \mu_2$.
Since both measures are conjugate-invariant,
$|h(z)| = |h(\bar{z})|$ almost everywhere.
Thus $W = M_{|h|}$ is unitary and intertwines
both $M_{\boldsymbol{z}}$ and $c$.
\end{proof}

For OSF representations,
we use the realization in Fact~\ref{fact:osffunctionalrealization}
to explicitly obtain the standard $2 \times 2$ block-operator representation of a pair of
orthogonal projections appearing in Halmos' two projections theorem \cite{PRHalmos1969}.
We follow the formulation and notation in
\cite[Theorem~1.1]{ABottcher2010}.

\begin{Proposition}[Explicit realization of Halmos' two projections theorem]
\label{proposition:halmoscmv}
Let $(\pi,V)$ be a nonzero OSF representation,
and let $\mu$ be the spectral measure of
$\pi(g_2g_1)$ with respect to $e_1$.
Define $P,Q$ by
\begin{equation*}
P= \frac{\pi(g_1) + \operatorname{id}_V}{2},
\quad
Q = \frac{\pi(g_2) + \operatorname{id}_V}{2}.
\end{equation*}

Identify $V$ with $L^2(S^1, \mu)$ through
the unitary operator given by
Fact~\ref{fact:osffunctionalrealization},
and define $V^\circ$ by
\begin{equation*}
V^\circ=L^2(S^1\setminus\{-1,1\},\mu).
\end{equation*}
Assume that $V^\circ\ne\{0\}$.

Then $M_0,M_1$ in
\cite[Theorem~1.1]{ABottcher2010} are given by
\begin{align*}
M_0 &= \{f\in V^\circ:cf=f\},
\ 
M_1 = \{f\in V^\circ:cf=-f\},
\end{align*}
and the unitary operator $R:M_1\to M_0$
and the self-adjoint operators
$C,S:M_0\to M_0$ in \cite[Theorem~1.1]{ABottcher2010}
can be chosen as
\begin{equation*}
(Rf)(z)
=-i\ \operatorname{sgn}(\operatorname{Im}(z))\ f(z),
\end{equation*}
\begin{equation*}
(Cf)(z)
=\sqrt{\frac{1+\operatorname{Re}(z)}{2}}\ f(z),
\qquad
(Sf)(z)
=\sqrt{\frac{1-\operatorname{Re}(z)}{2}}\ f(z).
\end{equation*}
\end{Proposition}
\begin{proof}
This follows from elementary calculations.
\end{proof}

\section{Characterizations of multiplicity-free representations}
\label{sec:multiplicityfreerepresentationcharacterization}

The following theorem shows that every multiplicity-free unitary
representation of $D_\infty$, when restricted to the orthogonal
complement of a subspace of dimension at most two,
is unitarily equivalent to an OSF representation.
For two unitary representations $\pi_1$ and $\pi_2$
of a locally compact group, we write $\pi_1 \le \pi_2$
if $\pi_1$ is unitarily equivalent to a subrepresentation of $\pi_2$.

\begin{Theorem}
\label{theorem:characterizationmultiplicityfreeinfinitedihedral}
Let $(\pi,V)$ be a unitary representation of $D_\infty$.
Then the following conditions are equivalent:
\begin{enumerate}
  \item[(i)] $\pi$ is multiplicity-free.
  \item[(ii)] There exists an invariant subspace $U$ of dimension at most two 
  such that $\pi_U \le \iota_{-,+} \oplus \iota_{-,-}$
  and $\pi_{U^{\perp}}$ is unitarily equivalent to an OSF representation,
  where $\pi_U$ and $\pi_{U^{\perp}}$ are subrepresentations corresponding to
  $U$ and $U^{\perp}$, respectively.
  \item[(iii)] There exists an invariant subspace $U$ of dimension at most two 
  such that $\pi_U \le \iota_{-,+} \oplus \iota_{-,-}$
  and $\iota_{-,+} \not\le \pi_{U^{\perp}}$
  and $\iota_{-,-} \not\le \pi_{U^{\perp}}$
  and $\pi_{U^{\perp}}(g_2g_1)$ has a cyclic vector in $U^{\perp}$.
\end{enumerate}
\end{Theorem}
\begin{proof}
The implication (i) $\implies$ (ii) follows directly from
Theorem~\ref{theorem:dinftyunitaryrepresentationdecomposition}.

The implication (ii) $\implies$ (iii) follows from
Theorem~\ref{theorem:explicitdirectintegralosf}
and the fact that every CMV matrix has a cyclic vector.

The cyclicity of $\pi_{U^\perp}(g_2g_1)$ implies that
$\pi_{U^\perp}$ is multiplicity-free.
Furthermore, $\pi_{U}$ is multiplicity-free.
Since $\pi_{U^{\perp}}$ and $\pi_U$ are disjoint unitary representations,
$\pi$ is multiplicity-free.
Hence we have (iii) $\implies$ (i).
\end{proof}

The following proposition provides several equivalent characterizations
of OSF representations.

\begin{Proposition}
\label{position:characterizationosf}
Let $(\pi,V)$ be a unitary representation of $D_\infty$.
Then the following conditions are equivalent:
\begin{enumerate}
  \item[(i)] $\pi$ is unitarily equivalent to an OSF representation.
  \item[(ii)] $\iota_{-,+} \not\le \pi$, $\iota_{-,-} \not\le \pi$,
  and $\pi$ is multiplicity-free.
  \item[(iii)] $\iota_{-,+} \not\le \pi$, $\iota_{-,-} \not\le \pi$,
  and there exists a vector $v$ of $V$ such that
  $\overline{\operatorname{span}(\pi(\mathbb{Z})v)}=V$.
  \item[(iv)] $\pi$ is a spherical representation
  of the Gelfand pair $(D_\infty,\langle g_1 \rangle)$.
\end{enumerate}
\end{Proposition}
\begin{proof}
From Theorem~\ref{theorem:characterizationmultiplicityfreeinfinitedihedral},
conditions (i), (ii), and (iii) are equivalent.
The equivalence between (i) and (iv) follows from Theorem~\ref{theorem:sphericalrepresentation}.
\end{proof}

The following two propositions characterize when the restriction of a
unitary representation $\pi$ of $D_\infty$ to $\mathbb{Z}$ is multiplicity-free.
Proposition~\ref{proposition:dinftyrestrictionmutiplicityfreeoperator}
characterizes this condition in terms of the commutant $\pi(D_\infty)'$,
while Proposition~\ref{proposition:dinftyrestrictionmutiplicityfreeverblunsky}
characterizes it in terms of Verblunsky coefficients.

\begin{Proposition}
\label{proposition:dinftyrestrictionmutiplicityfreeoperator}
Let $(\pi,V)$ be a unitary representation of $D_\infty$.
Then $\pi|\mathbb{Z}$ is multiplicity-free
if and only if
$\pi(D_\infty)'=W^*(\pi(g_2g_1)+\pi(g_1g_2))$.
Furthermore, if $\pi|\mathbb{Z}$ is multiplicity-free,
then there exists a unitary representation $\rho$ on $V$
such that $\rho$ is unitarily equivalent to an OSF representation,
$\rho|\mathbb{Z}=\pi|\mathbb{Z}$, and
$\pi(D_\infty)'=\rho(D_\infty)'$.
\end{Proposition}
\begin{proof}
First, assume that $\pi|\mathbb{Z}$ is multiplicity-free.
Since $\pi(D_\infty)' \subset \pi(\mathbb{Z})'$
and $\pi(\mathbb{Z})'$ is commutative,
$\pi$ is multiplicity-free.

Furthermore, from Theorem~\ref{theorem:characterizationmultiplicityfreeinfinitedihedral},
there exist a unitary representation $(\pi',H)$ of $D_\infty$ and
$a_{\pm} \in \{0,1\}$ and an OSF representation $\pi_0$ such that
\begin{equation*}
\pi \simeq \pi' = a_{+}\iota_{-,+} \oplus a_{-} \iota_{-,-} \oplus \pi_0.
\end{equation*}
Since $\pi|\mathbb{Z}$ is multiplicity-free,
$\iota_{-,\pm} \not\le \pi_0$ if $a_{\pm}>0$.

Define $(\pi_1,H)$ by
$\pi_1=a_+\iota_{+,-}\oplus a_-\iota_{+,+}\oplus\pi_0$.
Since $\pi'|\mathbb{Z}$ is multiplicity-free,
the nonzero one-dimensional summands of
$\pi_1$ are disjoint from $\pi_0$.
Consequently,
$\pi_1|\mathbb{Z}=\pi'|\mathbb{Z}$ and
$\pi_1(D_\infty)'=\pi'(D_\infty)'$.

Since $\pi_1|\mathbb{Z}=\pi'|\mathbb{Z}$ and $\pi'|\mathbb{Z}$
is multiplicity-free, $\pi_1$ is multiplicity-free.
Furthermore, we have $\iota_{-,\pm} \not\le \pi_0$ from
Theorem~\ref{theorem:explicitdirectintegralosf}.
From Proposition~\ref{position:characterizationosf},
$\pi_1$ is unitarily equivalent to an OSF representation.
Since $\pi_1$ is unitarily equivalent to an OSF representation,
Theorem~\ref{theorem:explicitdirectintegralosf} gives
\begin{equation*}
\pi_1(D_\infty)'
=W^*(\pi_1(g_2g_1)+\pi_1(g_1g_2)).
\end{equation*}
Since $\pi_1|\mathbb{Z}=\pi'|\mathbb{Z}$
and $\pi_1(D_\infty)'=\pi'(D_\infty)'$,
the same equality holds for $\pi'$.
By unitary equivalence, it also holds for $\pi$.

Let $F:V \to H$ be a unitary intertwining operator
from $\pi$ to $\pi'$.
Define a unitary representation $\rho$ on $V$ by
$\rho(g)=F^{-1}\pi_1(g)F \quad(g\in D_\infty)$.
Then $\rho$ is unitarily equivalent to an OSF
representation, and
\begin{equation*}
\rho|\mathbb{Z}=\pi|\mathbb{Z},
\qquad
\rho(D_\infty)'=\pi(D_\infty)'.
\end{equation*}

Conversely, let $C=\pi(g_2g_1)$ and $T=C+C^*$,
and assume that $\pi(D_\infty)'=W^*(T)$.
Then $\pi$ is multiplicity-free.
If $\pi|\mathbb{Z}$ were not multiplicity-free,
Theorem~\ref{theorem:characterizationmultiplicityfreeinfinitedihedral}
would imply that $C$ has an eigenvalue
$\lambda\in\{-1,1\}$ shared by an exceptional
one-dimensional summand and an OSF summand.
The projection onto the exceptional one-dimensional
summand
would then be in $\pi(D_\infty)'$ but not
in $W^*(T)$, since $T = 2\lambda \operatorname{id}_W$ on the
corresponding two-dimensional eigenspace $W$.
This is a contradiction.
\end{proof}

\begin{Proposition}
\label{proposition:dinftyrestrictionmutiplicityfreeverblunsky}
Let $(\pi,V)$ be a multiplicity-free unitary
representation of $D_\infty$,
and let $U$ be the invariant subspace given by
Theorem~\ref{theorem:characterizationmultiplicityfreeinfinitedihedral}.
Assume that $U^\perp\ne\{0\}$,
and let $\mathcal{A}$ be a sequence of real Verblunsky
coefficients of length $\dim U^\perp$
such that the corresponding real CMV matrix
is unitarily equivalent to $\pi_{U^\perp}(g_2g_1)$.
Define $N$ by $N=\dim U^{\perp}$.

When $N < \infty$, let $\gamma$ be the last component
of $\mathcal{A}$.
When $N = \infty$, define $\theta_j \in (0,\pi)$ by
$\mathcal{A} = (\cos\theta_j)_{j=0}^\infty$, and define $(a_m)_{m=1}^\infty$
and $(b_m)_{m=1}^\infty$ by
\begin{align*}
a_m&=
\frac{1}{\cos(\theta_{2m-2}/2)}
\prod_{j=1}^{m-1}
\tan(\theta_{2j-2}/2)\tan(\theta_{2j-1}/2),\\
b_m&=
\frac{(-1)^{m+1}}{\sin(\theta_{2m-2}/2)}
\prod_{j=1}^{m-1}
\frac{\tan(\theta_{2j-1}/2)}
{\tan(\theta_{2j-2}/2)}
\qquad(m\in\mathbb N).
\end{align*}

Then $\pi|\mathbb Z$ is not multiplicity-free
if and only if at least one of the following holds:
\begin{enumerate}
\item[(i)] $\iota_{-,-}\le\pi_U$, and either
$N<\infty$ with $\gamma=1$, or
$N=\infty$ with $(a_m)_{m=1}^\infty\in\ell^2$.
\item[(ii)] $\iota_{-,+}\le\pi_U$, and either
$N<\infty$ with $\gamma=(-1)^N$, or
$N=\infty$ with $(b_m)_{m=1}^\infty\in\ell^2$.
\end{enumerate}
\end{Proposition}
\begin{proof}
Since $\pi_{U^\perp}|\mathbb Z$ is multiplicity-free,
$\pi|\mathbb Z$ is not multiplicity-free
if and only if $\pi_U$ contains $\iota_{-,-}$
and $1$ is an eigenvalue of
$\pi_{U^\perp}(g_2g_1)$,
or $\pi_U$ contains $\iota_{-,+}$
and $-1$ is an eigenvalue of
$\pi_{U^\perp}(g_2g_1)$.

For $N < \infty$, a simple calculation shows that
$1$ (respectively, $-1$) is an eigenvalue of
$\pi_{U^\perp}(g_2g_1)$ if and only if
$\gamma=1$ (respectively, $\gamma=(-1)^N$).
For $N=\infty$, $1$ (respectively, $-1$) is an eigenvalue
if and only if $(a_m)_{m=1}^\infty \in \ell^2$
(respectively, $(b_m)_{m=1}^\infty \in \ell^2$);
see, e.g., \cite{BSimon2005p1}.
\end{proof}

The following proposition provides, in terms of real CMV matrices and spectral measures, 
necessary and sufficient conditions under which a multiplicity-free unitary representation 
of $\mathbb{Z}$ can be extended to a multiplicity-free unitary representation of $D_\infty$.

\begin{Proposition}
\label{proposition:dinftyextentionmutiplicityfree}
Let $(\pi,V)$ be a multiplicity-free unitary representation of $\mathbb{Z}$.
Then the following conditions are equivalent:
\begin{enumerate}
\item[(i)] There exists a unitary representation $\tilde{\pi}$ of $D_\infty$ such that $\tilde{\pi}|\mathbb{Z}\simeq\pi$.
\item[(ii)] $\pi(1)$ is unitarily equivalent to a real CMV matrix.
\item[(ii)'] There exists a conjugate-invariant positive regular Borel measure $\mu$ on $S^1$ such that
$\mu$ and a spectral measure of $\pi(1)$ with respect to a cyclic vector
are mutually absolutely continuous.
\item[(ii)''] $\pi(1)$ is unitarily equivalent to the multiplication operator $M_{\boldsymbol{z}}$ on $L^2(S^1,\mu)$,
where $\mu$ is a conjugate-invariant positive regular Borel measure on $S^1$.
\end{enumerate}
Note that if condition (i) is satisfied then there exists an OSF representation $\tilde{\pi}$
such that $\tilde{\pi}|\mathbb{Z} \simeq \pi$.
\end{Proposition}
\begin{proof}
From Proposition~\ref{proposition:dinftyrestrictionmutiplicityfreeoperator},
we have that (i) $\implies$ (ii).
From the spectral resolution theorem for unitary operators,
standard results on CMV matrices  (see, e.g., \cite{BSimon2005p1})
and Verblunsky's theorem \cite{SVerblunsky1935},
(ii), (ii)', and (ii)'' are equivalent.

Assume that $\pi(1)$ is unitarily equivalent to
a real CMV matrix $C$.
The natural factorization of $C$
yields an OSF representation $\tilde{\pi}$
of $D_\infty$ such that $\tilde{\pi}(g_2g_1) = C$.
Hence $\tilde{\pi}|{\mathbb{Z}} \simeq \pi$,
proving (ii) $\implies$ (i).
\end{proof}

\section*{Acknowledgements}
The author is grateful to Professor Takayuki Okuda for valuable suggestions and comments
on the direction and scope of this research.
The author would also like to thank Koichi Tojo, Shunsuke Miyauchi,
and Sota Takahashi for taking the time to discuss the results
and for their helpful comments and suggestions.

\section*{Declaration of generative AI and AI-assisted technologies
in the manuscript preparation process}

During the preparation of this manuscript, the author used
OpenAI's ChatGPT to assist with English-language editing,
improving the clarity and organization of the manuscript,
and discussing mathematical arguments and proofs.
The author carefully reviewed and verified the AI-assisted
suggestions, edited the manuscript as needed, and takes
full responsibility for its content.

\appendix

\section{Proof of Fact~\ref{fact:sphericalcriteria}}
\label{sec:proofofsphericalcriteria}

We give a proof of Fact~\ref{fact:sphericalcriteria}.

\begin{proof}
For each $\omega \in \widehat{G}^{K\text{-}\mathrm{sph}}$,
let $P_\omega$ be the orthogonal projection
onto the space of $K$-fixed vectors, obtained by averaging
$\pi_\omega$ over $K$.
Applying this family of nonzero projections
to a fundamental sequence, and normalizing the first nonzero
image at each point, gives a measurable field of
$K$-fixed unit vectors $v_\omega$ $(\omega \in \widehat{G}^{K\text{-}\mathrm{sph}})$.
Since $m$ is finite, $(v_\omega)_{\omega \in \widehat{G}^{K\text{-}\mathrm{sph}}}$
is in the direct integral Hilbert space and is $K$-fixed.
Define $v$ by $v=(v_\omega)_{\omega \in \widehat{G}^{K\text{-}\mathrm{sph}}}$.

Let $\mathcal{M}$ be the closed cyclic subspace generated by $v$,
and let $R$ be the orthogonal projection onto $\mathcal M$.
Then $R \in \pi(G)'$.
Since $\pi$ is multiplicity-free, $\pi(G)'$ is commutative.
Furthermore, it contains every diagonal multiplication operator
$M_f$ $(f \in L^\infty(\widehat{G}^{K\text{-}\mathrm{sph}}, m))$.
Hence $R$ commutes with $M_f$, so
$M_f\pi(g)v \in \mathcal{M}$
for any $f \in L^\infty(\widehat{G}^{K\text{-}\mathrm{sph}}, m)$ and $g\in G$.

Let $g \in G$.
If $h \in \mathcal M^\perp$, orthogonality to all these vectors gives
\begin{equation*}
(h_\omega, \pi_\omega(g)v_\omega) = 0
\quad\text{for almost every }\omega.
\end{equation*}
Applying this to a countable dense subset of $G$,
and then using the continuity of each $\pi_\omega$, gives the equality
for any $g\in G$ outside a common null set.
Since each $\pi_\omega$ is irreducible and $v_\omega\ne0$,
we obtain $h_\omega=0$ almost everywhere.
Thus $\mathcal M^\perp=\{0\}$. It follows that $v$ is cyclic.
\end{proof}

\end{document}